\documentclass[11pt]{amsart}

\usepackage{amsmath, amssymb, mathtools, palatino, euler, epic, eepic,floatflt, microtype}
\usepackage{graphicx}
\usepackage[usenames]{xcolor}
\usepackage{float}

\usepackage[margin=1 in]{geometry}
\definecolor{darkblue}{rgb}{0,0,0.7}
\definecolor{darkred}{rgb}{0.7,0,0}
\usepackage[colorlinks,linkcolor=darkred, citecolor=darkblue,
pagebackref=true,pdftex]{hyperref}

\newcommand\defi[1]{\textit{\color{blue}#1}}    

\renewcommand\gcd{\operatorname{gcd}}
\newcommand\pdim{\operatorname{pdim}}
\newcommand\rank{\operatorname{rank}}
\newcommand\Tor{\operatorname{Tor}}

\newtheorem{thm}{Theorem}[section]
\newtheorem{prop}[thm]{Proposition}
\newtheorem{lemma}[thm]{Lemma}
\newtheorem{conj}[thm]{Conjecture}
\newtheorem{cor}[thm]{Corollary}

\newtheorem{defn}[thm]{Definition}

\newtheorem{rem}[thm]{Remark}
\newtheorem{example}[thm]{Example}

\begin{document}

\title{Exposed circuits, linear quotients, and chordal clutters}

\author{Anton Dochtermann}
\address{Department of Mathematics, Texas State University, San Marcos}
\email{dochtermann@txstate.edu}



\date{\today}

\begin{abstract}
A graph $G$ is said to be chordal if it has no induced cycles of length four or more. In a recent preprint Culbertson, Guralnik, and Stiller give a new characterization of chordal graphs in terms of sequences of what they call `edge-erasures'.  We show that these moves are in fact equivalent to a linear quotient ordering on $I_{\overline{G}}$, the edge ideal of the complement graph. Known results imply that $I_{\overline G}$ has linear quotients if and only if $G$ is chordal, and hence this recovers an algebraic proof of their characterization.  We investigate higher-dimensional analogues of this result, and show that in fact linear quotients for more general circuit ideals of $d$-clutters can be characterized in terms of removing exposed circuits in the complement clutter.  Restricting to properly exposed circuits can be characterized by a homological condition. This leads to a notion of higher dimensional chordal clutters which borrows from commutative algebra and simple homotopy theory.  The interpretation of linear quotients in terms of shellability of simplicial complexes also has applications to a conjecture of Simon regarding the extendable shellability of $k$-skeleta of simplices.  Other connections to combinatorial commutative algebra, chordal complexes, and hierarchical clustering algorithms are explored.
\end{abstract}

\maketitle
\section{Introduction}
Chordal graphs are a widely studied class of combinatorial objects, with connections to various algorithmic and structural questions and generalizations in a variety of directions.  Perhaps a major reason for their wide appeal is their various characterizations in terms of seemingly unrelated properties, incorporating topological, combinatorial, and algebraic notions.  For instance the clique complex of a chordal graph has collapsible components, whereas the independence complex of a chordal graph is known to be vertex decomposable \cite{DocEng}, which in particular implies that is has the homotopy type of a wedge of spheres.

Recently in \cite{CulGurSti} a new characterization of chordal graphs was given in terms of performing a series of `edge-erasures' on a complete graph.  For this we say that an edge $e$ is a graph $G$ is \emph{exposed} if $e$ is contained in a unique maximal clique $K$ of $G$.  We say that $e$ is \emph{properly exposed} if $|K| > 2$  (i.e. $e$ is contained in some triangle).  If $G$ is a graph and $e \in G$ is properly exposed we say that $G - \{e\}$ is obtained from $G$ via an  \emph{edge erasure}.  With this notation the authors of \cite{CulGurSti} prove the following.

\begin{thm}[\cite{CulGurSti}]\label{thm:CGS}
A connected graph $G$ is chordal if and only if $G$ can be obtained from a complete graph by a sequence of edge erasures.
\end{thm}

As the authors point out, this description of a chordal graph in terms of sequences of edges has a different flavor than other characterizations in terms of simplicial neighborhoods of vertices, etc. In \cite{CulGurSti} this characterization is used to give a new algorithm for finding a minimum spanning tree in a finite metric space, a modified version of the greedy algorithm due to Kruskal.   The notion of an exposed edge is reminiscent of the \emph{elementary collapses} from simple homotopy theory and in particular makes sense in the more general context of hypergraphs and $d$-clutters. We discuss this further below.

Chordal graphs also make an appearance in the context of combinatorial commutative algebra.  Suppose $G = (V,E)$ is a graph on vertex set $V = [n] = \{1,2, \dots, n\}$. For a fixed field ${\mathbb K}$ one can construct the \emph{edge ideal} $I_G$ in the polynomial ring $R = {\mathbb K}[x_1,x_2, \dots, x_n]$.  By definition $I_G$ is the monomial ideal generated by quadratic monomials corresponding to edges of the $G$:
\[I_G = \langle x_ix_j: ij \in E(G) \rangle.\]
We note that any squarefree quadratic monomial ideal $I$ can be realized as the edge ideal of some graph.  Typical research questions involve studying how algebraic properties of $I_G$ relate to combinatorial properties of the underlying graph $G$.

A well-known theorem of Fr\"oberg \cite{Fro} characterizes chordal graphs in terms of an \emph{algebraic} property of the associated edge ideal:  a graph $G$ is chordal if and only if t $I_{\overline G}$, the edge ideal of the complement graph, has a \emph{linear resolution}.  The property of having a linear resolution describes a particularly low `complexity' in the relations among the generators of $I_G$ (along with the relations among the relations, etc.).  More recently in \cite{HerHibZhe} it has been shown that if an edge ideal $I_G$ has a linear resolution then in fact it has \emph{linear quotients}, a condition that is stronger in the case of more general ideals.   To say that an ideal $I$ has linear quotients means that there exists an ordering of the generators $I = \langle m_1, m_2, \dots, m_k \rangle$ of the ideal such that each colon ideal $(I_{j}:m_{j+1})$ is generated by a collection of linear forms.  Here $I_j = \langle m_1, m_2, \dots, m_j \rangle$. More details regarding these algebraic concepts are given in the next section.

The notion of an edge ideal of a graph generalizes to the context of \emph{$d$-clutters} (uniform hypergraphs where the edge set consists of $d$-subsets of $[n]$), where generators again correspond to circuits (the `edges' of the $d$-clutter). If $e$ is a circuit of a $d$-clutter ${\mathcal C}$ we use ${\bf x}_e$ to denote the squarefree monomial that it defines.   The notion of a exposed edge also generalizes : if  $e \in {\mathcal C}$ is a circuit of some $d$-clutter ${\mathcal C}$ on vertex set $[n] = \{1,2, \dots, n\}$  we say that $e$ is \emph{exposed} if it is uniquely contained in some maximal $d$-clique $K$.   We say that $e$ is \emph{properly exposed} if $|K| > d$. Our main result says that removing an exposed circuit from a $d$-clutter corresponds to adding a generator to the circuit ideal of the complement clutter that satisfies a particular algebraic property.

\newtheorem*{thm:main}{Theorem \ref{thm:main}}
\begin{thm:main}
	Suppose ${\mathcal C}$ is a $d$-clutter and let $e \in {\mathcal C}$ be a circuit in ${\mathcal C}$.   Then $e$ is an exposed circuit if and only if ${\bf x}_e$ is a linear divisor for the ideal $I_{\overline{\mathcal C}}$, where $\overline{\mathcal C}$ is the complement of ${\mathcal C}$. Moreover $e$ is contained in a unique maximal clique $K \subset {\mathcal C}$ if and only if the colon ideal $(I_{\overline{\mathcal C}}: {\bf x}_e)$ is generated by variables corresponding to vertices in the complement of $K$.
\end{thm:main}

As a consequence we obtain an algebraic proof of one part of Theorem \ref{thm:CGS}, namely that a graph $G$ is obtained from a complete graph through a sequence of removing exposed edges if and only if $I_{\overline G}$, the edge ideal of the complement graph, has linear quotients.  The result of \cite{HerHibZhe} then implies that this is the case if and only $G$ is chordal.   In addition, it is not hard to show that a chordal graph $G$ obtained from a sequence of exposed edges is \emph{connected} if and only if each edge in the sequence is \emph{properly} exposed.  The property of a graph $G$ being connected also has an algebraic interpretation in terms of the Betti table of the underlying edge ideal $I_{\overline{G}}$.  This generalizes to the context of circuit ideals of clutters where we establish the following higher dimensional analogue of Theorem \ref{thm:CGS}.

\newtheorem*{prop:cluttererasure}{Proposition \ref{prop:cluttererasure}}
\begin{prop:cluttererasure}
	A $d$-clutter ${\mathcal C}$ can be obtained from the complete $d$-clutter $K_n^d$ through a sequence of circuit erasures if and only if $I_{\overline{\mathcal C}}$ has linear quotients and 
	\[\pdim(I_{\overline {\mathcal C}}) < n-d.\]
\end{prop:cluttererasure}

Here $\pdim$ denotes the projective dimension, see below for a definition. Our result also provides a formula for the Betti numbers of an ideal with linear quotients (generated in a fixed degree) in terms of the combinatorics of the exposed faces removed in the complement, see Corollary \ref{cor:Betti}.

It is known that squarefree ideals with linear quotients are strongly related (via Alexander duality) to the notion of \emph{shellability} for a simplicial complex. A shellable simplicial complex $\Delta$ is said to be \defi{extendably shellable} if every shelling of a subcomplex of $\Delta$ can be continued to a shelling of $\Delta$.  Not all shellable complexes are extendably shellable (for instance certain $d$-dimensional simplicial spheres for $d \geq 3$, as discussed in \cite{Zie}) but a conjecture of Simon \cite{Sim} says that all $k$-skeleta of a simplex on $[n]$ are extendably shellable.  In Section \ref{sec:extend} we show how our results lead to a proof of this conjecture for the case $k \geq n-3$ (which was also obtained recently in \cite{BPZDec} using other methods).

\newtheorem*{cor:shell}{Corollary \ref{cor:shell}}
\begin{cor:shell}
For all $k \geq n-3$, the $k$-skeleton of a simplex on vertex set $[n]$ is extendably shellable.
\end{cor:shell}

As we have seen, a sequence of deleting (properly) exposed edges from a complete graph gives a characterization of (connected) chordal graphs.  Hence Theorem \ref{thm:main} provides a candidate for a notion of a higher dimensional `chordal complex' which borrows from simple homotopy and combinatorial commutative algebra.  In recent years several authors have introduced (mostly independent) notions of chordal complexes which generalize the various characterizations of chordal graphs to higher dimensions.  As far as we know the direct connection to free faces and elementary collapses has not been considered, although the recent preprint \cite{BigFar} explores similar territory. We briefly discuss these approaches in Section \ref{sec:higher} where we also offer a conjectural connection to the constructions discussed here.

The rest of the paper is organized as follows.  In Section \ref{sec:definitions} we review relevant definitions, including basic notions from clutter theory and combinatorial commutative algebra. In Section \ref{sec:results} we prove the results mentioned above and discuss some further corollaries and examples.  In Section \ref{sec:higher} we discuss applications to shellability and higher dimensional notions of chordal complexes. We end with some discussion regarding connections to data clustering (the original motivation for \cite{CulGurSti}), as well as some open problems.

{\bf Acknowledgements.}  We wish to thank Mina Bigdeli for helpful comments and correspondence, in particular regarding the connection to simplicial ridges.  We also thank Davide Bolognini, Sara Faridi, Jared Culbertson, Dan Guralnik, and Peter Stiller for fruitful conversations.  We are grateful to the two anonymous referees for their careful reading and the many suggestions that helped to improve the paper. Macaulay2 \cite{Mac} was used extensively to compute examples and we have included calculations in figures below.

\section{Definition and objects of study} \label{sec:definitions}

\subsection{Clutters and simplicial complexes}
We begin by recalling some relevant combinatorial notions. Recall that a \defi{clutter} on vertex set $[n] = \{1,2,\dots, n\}$ is a collection ${\mathcal C}$ of subsets of $[n]$, none of which properly contain another.  In this context the elements of ${\mathcal C}$ are called \defi{circuits}. In this paper we will usually restrict our attention to case where all circuits have the same size $d$ for some integer $d \geq 1$, in which case ${\mathcal C}$ will be called a \defi{$d$-clutter} (also sometimes called a $d$-uniform hypergraph). Note that a (simple) graph is the same as a $2$-clutter.  The \defi{complement} of a $d$-clutter ${\mathcal C}$, denoted $\overline {\mathcal C}$, is the $d$-clutter on the same vertex set $[n]$, where a $d$-subset $S \subset [n]$ is a circuit in $\overline {\mathcal C}$ if and only if $S \notin {\mathcal C}$.  An \defi{independent set} of ${\mathcal C}$ is a subset of $[n]$ containing no circuit.  For any integer $d \geq 2$ we use $K_n^d$ to denote the \defi{complete $d$-clutter} on vertex set $[n]$, which by definition consists of all $d$-subsets of $[n]$.  It is customary to use $K_n$ to denote $K_n^2$, the \emph{complete graph}.

If ${\mathcal C}$ is a $d$-clutter then a \defi{$d$-clique} (or just clique if the context is clear) is a nonempty collection of vertices $S \subset [n]$ with the property that $|S| < d$ or if $|S| \geq d$ every $d$ subset of $S$ is a circuit of ${\mathcal C}$.  The clique is said to be \defi{maximal} if $S$ is maximal with this property. If ${\mathcal C}$ is a $d$-clutter and $e \in {\mathcal C}$ is a circuit we say that $e$ is \defi{exposed} if $e$ is contained in a unique maximal clique $S$ in ${\mathcal C}$.  We say that $e$ is \defi{properly exposed} if $|S| > d$. 

A \defi{simplicial complex} $\Delta$ on vertex set $[n] = \{1,2,\dots, n\}$ is a collection of subsets of $[n]$ (called the \defi{faces} of $\Delta$) with the property that any subset of a face of $\Delta$ is itself a face of $\Delta$.  If a face $F \in \Delta$ has cardinality $d+1$ we say that has \emph{dimension} $d$ (and call it a $d$-face).  A subset $\sigma \subset [n]$ is a minimal non-face of $\Delta$ if $\sigma \notin \Delta$, but any proper subset of $\sigma$ is a face of $\Delta$. The maximal faces of $\Delta$ (under inclusion of sets) are called \defi{facets}, and $\Delta$ is said to be \defi{pure} if all facets have the same dimension.  The \defi{Alexander dual} of $\Delta$ is the simplicial complex $\Delta^\vee$ on vertex set $[n]$ with faces given by
\[\Delta^\vee = \{\sigma \subset [n]: [n] \backslash \sigma \notin \Delta\}.\]
\noindent
In particular the facets of $\Delta^\vee$ are given by the complements of minimal non-faces of $\Delta$.

A pure $d$-dimensional simplicial complex $\Delta$ is said to be \defi{shellable} if there is an ordering of the facets $F_1, F_2, \dots, F_s$ such that for all $k = 2,3, \dots, n$ the simplicial complex induced by
\[\big(\bigcup_{i=1}^{k-1} F_i  \big ) \cap F_ k\]
\noindent
is pure of dimension $d-1$.

Note that (arbitrary) clutters and simplicial complexes are related via the following constructions (we follow the conventions of \cite{Woo}).  For any clutter ${\mathcal C}$ on vertex set $[n]$ let
\[I({\mathcal C}) = \{\sigma \subset [n]: \text{$\sigma$ is an independent set of ${\mathcal C}$}\}\]
\noindent
denote the \defi{independence complex} of ${\mathcal C}$.  For any simplicial complex $\Delta$, let ${\mathcal C}(\Delta)$ denote the clutter consisting of all minimal non-faces of $\Delta$.  Then one can check that
\[{\mathcal C}\big(I({\mathcal C})\big) = {\mathcal C} \hspace{.15 in} \text{and} \hspace{.15 in} I\big({\mathcal C}(\Delta)\big) = \Delta.\]

\subsection{Circuit ideals and linear quotients}

Next we recall some relevant notions from commutative algebra.  We will fix a field ${\mathbb K}$ and let $R = {\mathbb K}[x_1, x_2, \dots, x_n]$ denote the polynomial ring on $n$ variables.   A $d$-clutter ${\mathcal C}$ on vertex set $[n] = \{1, 2, \dots, n\}$ naturally gives rise to a monomial ideal in $R$.  For this if $e = \{v_1, v_2, \dots, v_d\} \subset [n]$ is any subset of the vertex set we let
\[{\bf x}_e = x_{v_1}x_{v_2} \cdots x_{v_d}\]
denote the corresponding monomial in $R$.  We then let $I_{\mathcal C}$ denote the \defi{circuit ideal} of ${\mathcal C}$, generated by all such monomials corresponding to circuits of ${\mathcal C}$:
\[I_{\mathcal C} = \langle {\bf x}_e : e \in {\mathcal C} \rangle.\]
When $d=2$ we often say that $I_{\mathcal C}$ is the \emph{edge ideal} of the underlying graph ${\mathcal C}$. We note that any squarefree monomial ideal generated in degree $d$ can be thought of as the circuit ideal of a $d$-clutter (and vice versa) so these concepts are equivalent.  Quadratic squarefree monomial ideals are precisely the edge ideals of (simple) graphs.

We will be interested in homological properties of circuit ideals.  Given a graded ideal (or more generally a graded $R$-module) $I$, a \defi{free resolution of $I$} is an exact sequence
\begin{equation} \label{eq:res}
0 \leftarrow I \leftarrow F_0 \leftarrow F_1 \leftarrow \cdots \leftarrow F_p,
\end{equation}
where each
\[F_i = \bigoplus_{j \in {\mathbb Z}} R(-j)^{\beta_{i,j}}\]
is a free $R$-module and each map is a homogeneous module homomorphisms.  Here $R(-j)$ indicates the ring $R$ with the shifted grading, so that for all $a \in {\mathbb Z}$ we have
\[R(-j)_a = R_{a-j}.\]

\noindent
Note that replacing the last two maps in Equation \ref{eq:res} with $0 \leftarrow R/I \leftarrow R \leftarrow F_0$ provides a resolution of the quotient ring $R/I$, so we will sometimes move between the two notions.

 The resolution is said to be \defi{minimal} if the rank of each $F_i$ is minimum among all resolutions of $I$.  In this case we have $\beta_{i,j} = \beta_{i,j}(I) = \Tor_i^R(I,{\mathbb K})_j$, and these integers are called the \defi{graded Betti numbers} of $I$. The ordinary $i$th Betti number is given by $\beta_i = \sum_{j \in {\mathbb Z}} \beta_{i,j}$.   The \defi{projective dimension} of $I$, denoted $\pdim(I)$, is given by the length of a (and hence any) minimal resolution, so that
 \[\pdim(I) = \max \{i: \beta_i(I) \neq 0\}.\]

 For each $i =2,3,\dots p$, we can think of the maps $\partial_i:F_i \rightarrow F_{i-1}$ as matrices with entires in $R$, and the ideal $I$ is said to have a \defi{linear resolution} if all entries are linear forms.  If $I$ is generated in degree $d$ this is equivalent to having $\beta_{i,j} = 0$ whenever $i,j$ satisfy $j \neq i+d$.  We will often think of homological properties of an ideal that are preserved as we add one generator at a time.  For this we need the following notion.

\begin{defn}
Suppose $I \subset R$ is a monomial ideal generated, and suppose ${\bf x}_e$ is a squarefree monomial that is not a generator of $I$. Then we say ${\bf x}_e$ is a \emph{linear divisor} for $I$ if the colon ideal
\[(I:{\bf x}_e) = \{r \in R: r{\bf x}_e \in I \}\]
is generated by a subset of the variables $\{x_1, x_2, \dots, x_n\}$.
\end{defn}

\begin{defn}
A monomial $I$ is said to have \defi{linear quotients} if there exists an ordering of its generators $(m_1, m_2, \dots, m_g)$ such that $m_{j+1}$ is a linear divisor for $I_j$ for all $j = 1,2, \dots g-1$.

\noindent
Here for $j = 1, \dots, n$ we use the notation $I_j = \langle m_1, m_2, \dots, m_j \rangle$.
\end{defn}

The notion of an ideal with linear quotients was introduced by Herzog and Takayama in \cite{HerTak}.  The concept makes sense for arbitrary monomials ideals but here we will restrict ourselves to those that are squarefree and generated in a fixed degree (arising as the circuit ideal of some $d$-clutter ${\mathcal C}$).  Examples of such ideals include squarefree stable ideals as well as ideals generated by a collection of monomials whose support form the bases of a matroid.   

\begin{figure}[h] 
\includegraphics[scale = 0.425]{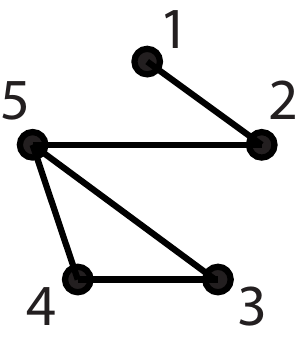}
\caption{A chordal graph $G$.}
\label{fig:example}
\end{figure}

 \begin{example} \label{ex:chordal}
 
 For a specific example consider the graph (2-clutter) $G$ depicted in Figure \ref{fig:example}.  The edge ideal of the complement graph $\overline G$ is given by
 \[I_{\overline G} = \langle x_1x_3, x_1x_4, x_1x_5, x_2x_3, x_2x_4 \rangle. \]
 
 \noindent
One can check that this ordering of the generators is in fact a linear quotient ordering for $I_{\overline G}$. For instance we have $I_4 = \langle x_1x_3, x_1x_4, x_1x_5, x_2x_3 \rangle$, $m_5 = x_2x_4$, and 
\[(I_4:m_5) =  \langle x_1, x_3 \rangle.\]

 \end{example}
 
Squarefree monomial ideals with linear quotients are closely related to shellable simplicial complexes, as the next observation indicates.  Here for a face $F \subset [n]$ we use $\overline F$ to denote the complement set, so that $\overline F = [n] \backslash F$.
 
 \begin{prop}\cite[Proposition 8.2.5]{HerHib} \label{prop:lin}
 Suppose $\Delta$ is a pure simplicial complex on the vertex set $[n]$. Then $F_1, F_2, \dots, F_s$ is a shelling order for $\Delta$ if and only if the ideal $\langle  {\bf x}_{\overline{F_1}}, {\bf x}_{\overline{F_2}}, \dots, {\bf x}_{\overline{F_s}} \rangle$ has linear quotients with respect to the given order.
 \end{prop}
 
 One can check that the ideal $\langle  {\bf x}_{\overline{F_1}}, {\bf x}_{\overline{F_2}}, \dots, {\bf x}_{\overline{F_s}} \rangle = I_{\Delta^\vee}$, where $I_{\Delta^\vee}$ denotes the \emph{Alexander dual} of the Stanley-Reisner ideal of $\Delta$.

In \cite{HerTak} Herzog and Takayama study minimal resolutions of monomial ideals with linear quotients.  To describe their construction suppose that $I$ is a monomial ideal with linear quotients for some ordering $(m_1, m_2  \dots, m_k)$ of the generators. A minimal resolution of $I$ is obtained by iteratively constructing mapping cones as follows.  Each time we add a generator $m_{j+1}$ we have a short exact sequence of $R$-modules
\[0 \rightarrow R/(I_j: m_{j+1}) \rightarrow R/I_j \rightarrow R/I_{j+1} \rightarrow 0,\]
\noindent
where $I_j = \langle m_1, m_2, \dots, m_j \rangle$ is the ideal generated by the first $j$ monomials in our (ordered) generating set. By assumption the colon ideal $(I_j:m_{j+1})$ is generated by some subset of the variables, say $\{x_{j_1}, x_{j_2}, \dots, x_{j_\ell}\}$. Hence  $R/(I_j: m_{j+1})$ has a minimal resolution given by a Koszul complex ${\mathcal K}_\ell$ on $\ell$ generators.   Assuming we have a minimal resolution ${\mathcal F}$ for the ideal $I_j$ we obtain a minimal resolution of $I_{j+1}$ by constructing the mapping cone $C(f)$ for the map of complexes $f:{\mathcal K}_\ell \rightarrow {\mathcal F}$ induced by the short exact sequence above.  Recall that the complex $C(f)$ is given by
\[C(f) = {\mathcal K}_\ell[1] \oplus {\mathcal F},\]
so that the module in the $i$th homology degree of $C(f)$ has rank given by 
\begin{equation} \label {eq:betti}
\rank F_i + {\ell \choose i},
\end{equation}
\noindent
where $F_i$ is the module in the $i$th homology degree of the complex ${\mathcal F}$. A \emph{cellular} realization for this mapping cone construction (under some further conditions on the ideal) was described in \cite{DocMoh}. 

\begin{example}
Returning to the ideal discussed in Example \ref{ex:chordal} we have $I_4 = \langle x_1x_3, x_1x_4, x_1x_5, x_2x_3 \rangle$, which has a minimal resolution given by
\[{\mathcal F} = \; 0 \leftarrow I_4 \leftarrow R^4 \leftarrow R^4 \leftarrow R \leftarrow 0.\]
If we add the generator $m_5 = x_2x_4$ we see that the colon ideal $(I_4 : m_5) = \langle x_1, x_3 \rangle$ is generated by two variables and hence has a minimal resolution given by a Koszul complex
\[{\mathcal K} = \; 0 \leftarrow (I_4 : m_5) \leftarrow R^2 \leftarrow R \leftarrow 0.\]

\noindent
Taking the mapping cone of the map of complexes ${\mathcal K} \rightarrow {\mathcal F}$ induced by the inclusion $m_5 \rightarrow I_5$ we obtain a minimal resolution of $I_5 = I_{\overline{G}}$ given by
\[0 \leftarrow I_{\overline{G}} \leftarrow R^5 \leftarrow R^6 \leftarrow R^2 \leftarrow 0.\]

\end{example}

\section{Main results and discussion}\label{sec:results}

In this section we provide proofs of the results discussed in the introduction.  Recall that a $d$-clutter ${\mathcal C}$ on vertex set $[n]$ gives rise to a monomial ideal $I_{\overline {\mathcal C}} \subset {\mathbb K}[x_1, x_2, \dots, x_n]$, generated by all squarefree monomials of degree $d$ not appearing in ${\mathcal C}$.    Removing a circuit from ${\mathcal C}$ corresponds to adding a generator to $I_{\overline {\mathcal C}}$.  We then have the following.

\begin{thm} \label{thm:main}
Suppose ${\mathcal C}$ is a $d$-clutter for $d \geq 1$ and let $e \in {\mathcal C}$ be a circuit in ${\mathcal C}$.   Then $e$ is an exposed circuit if and only if ${\bf x}_e$ is a linear divisor for the ideal $I_{\overline{\mathcal C}}$, where $\overline{\mathcal C}$ is the complement of ${\mathcal C}$.  Moreover $e$ is contained in a unique maximal clique $K$ if and only if the colon ideal $(I_{\overline{\mathcal C}}: {\bf x}_e)$ is generated by variables corresponding to the vertices $[n] \backslash K$.
\end{thm}

\begin{proof}
Suppose ${\mathcal C}$ is a clutter on vertex set $[n] = \{1,2,\dots, n\}$, and let $I_{\overline {\mathcal C}} \subset R = {\mathbb K}[x_1, x_2, \dots, x_n]$ denote the circuit ideal of its complement .  Suppose $e =\{v_1, v_2, \dots, v_d\}$ is a circuit in ${\mathcal C}$.

For one direction of the theorem suppose $e$ is an exposed circuit, so that $e$ is contained in a unique maximal $d$-clique of ${\mathcal C}$.  Without loss of generality suppose the $d$-clique consists of the vertices $K = \{1,2, \dots, k\}$. Let $I = I_{\overline {\mathcal C}}$ denote the edge ideal of the complement clutter $\overline {\mathcal C}$, and let $I_e = \langle I, {\bf x}_e \rangle$ denote the ideal obtained by adding the monomial ${\bf x}_e$  as another generator. We then have the inclusion $I \rightarrow I_e$.

We claim that the colon ideal
\[(I:{\bf x}_e) = \{r \in R: r{\bf x}_e \in I\}\]
is generated by the variables $X = \{x_{k+1}, \dots, x_n \}$, corresponding to variables \emph{not} in the clique $K$. To see this first note that if $x_\ell \in X$ then $\ell \cup (K \backslash \{i\})$ must \emph{not} contain a circuit of ${\mathcal C}$ for some $i = 1, 2, \dots, k$  (otherwise since $e$ is a circuit we would obtain a clique $e \cup \{\ell\}$ in ${\mathcal C}$ of size $d+1$, meaning that we could either add $\ell$ to $K$ to obtain a larger clique or else have that $e$ is contained in two distinct maximal cliques - either way a contradiction).  Without loss of generality suppose $\{\ell, v_2, v_3, \dots, v_d\}$ is missing from ${\mathcal C}$, so that $x_\ell x_{v_2} x_{v_3} \cdots x_{v_d} \in I$ and hence $x_\ell {\bf x}_e \in I$.
 We conclude that $x_\ell \in (I:{\bf x}_e)$, and hence $\langle x_{k+1}, \dots, x_n \rangle \subset (I:{\bf x}_e)$.
 
 Next suppose $m \in (I:{\bf x}_e)$, so that $m{\bf x}_e \in I$.  We claim that $m$ is contained in the ideal generated by the variables $X = \{x_{k+1}, \dots, x_n \}$. For this we will use the fact that since $I$ and $\langle {\bf x}_e\rangle$ are both monomial, the colon ideal $(I:{\bf x}_e)$ is also monomial (\cite[Proposition 1.2.2]{HerHib}).  In fact a (possibly redundant) set of generators of $(I:{\bf x}_e)$ is given by the collection
 \[\{u/ \gcd (u,{\bf x}_e): u \in G(I) \}.\]
Here $G(I)$ denotes a set of generators of $I$. To show that $m$ is contained in the desired ideal it's enough to show that each such generator of $(I:{\bf x}_e)$ contains some variable from among $\{x_{k+1}, \dots, x_n\}$.  But recall that $K = \{1,2,\dots, k\}$ is a $d$-clique and every generator $u \in G(I)$ is given by noncircuits.  Hence every generator $u \in G(I)$ contains some variable among $\{x_{k+1}, \dots, x_n\}$.  Since $\{v_1, v_2, \dots, v_d \} \subset \{1,2, \dots, k\}$ we have that $u/\gcd(u,{\bf x}_e)$ must also contain that variable.  We conclude that every generator of $(I:{\bf x}_e)$ contains some element among the variables $X$, and hence $(I:{\bf x}_e) \subset \langle x_{k+1}, \dots, x_n \rangle$.  

For the other direction suppose ${\bf x}_e = x_{v_1}x_{v_2} \cdots x_{v_d}$  is a linear divisor for the ideal $I = I_{\overline {\mathcal C}}$, so that the colon ideal $(I: {\bf x}_e)$ is generated by a subset of the variables.  Without loss of generality suppose
\[(I:{\bf x}_e) = \langle x_1, x_2, \dots, x_k \rangle.\]
We claim that the vertex set $S = \{k+1, k+2, \dots, n\}$ forms a maximal $d$-clique in the clutter ${\mathcal C}$, and that the circuit $e = \{v_1, v_2, \dots, v_d\}$ is uniquely contained in this clique. For this suppose $W = \{w_1, w_2, \dots, w_d\} \subset S$ and for a contradiction suppose $W$ did \emph{not} form a $d$-clique.  Then ${\bf x}_W$ would be a generator of $I$.  So then ${\bf x}_W {\bf x}_e \in I$ and hence ${\bf x}_W \in (I: {\bf x}_e) = \langle x_1, x_2, \dots, x_k \rangle$, a contradiction.

To show that $S$ is maximal suppose $a \leq k$ with the property that every $d$ subset of $\{a\} \cup S$ forms a circuit of ${\mathcal C}$. Then we have that $x_a{\bf x}_W$ is not a generator of $I = I_{\overline {\mathcal C}}$ for all $W \subset S$ with $|W| = d-1$.   But since $x_a$ is a generator of $(I:{\bf x}_e)$ we have $x_a {\bf x}_e \in I$ so that $x_a{\bf x}_W$ is a generator of $I$ for some $W \subset \{v_1, v_2, \dots, v_d\} \subset S$ with $|W| = d-1$, a contradiction.

Finally we show that $e$ is uniquely contained in $S$. For this suppose that $e = \{v_1, v_2, \dots, v_d\}$ is contained in some other maximal clique $T$, distinct from $S$.  Then there must be some vertex $t \in T$ (so that $t \leq k$) such that every $d$-subset of $\{t\} \cup \{v_1, v_2, \dots, v_d\}$ is a circuit in ${\mathcal C}$. But $x_t$ is a generator of $(I:{\bf x}_e)$ so that $x_t {\bf x}_W$ is a generator of $I$ for some for some subset $W \subset S$, again a contradiction.  The result follows.
\end{proof}

As mentioned in the introduction, it is known that if $\Delta$ is a simplicial complex then $I_\Delta$, the Stanley-Reisner ideal of $\Delta$, has linear quotients if and only if the Alexander dual $I^\vee$ has a shellable Stanley-Reisner complex (see also Proposition \ref{prop:lin}).  Presumably one could use this characterization to give a more combinatorial proof of Theorem \ref{thm:main}.  We will return to the connection to shellability in Section \ref{sec:applications}.

Note that in the case of a linear quotient ordering we are building the ideal one generator at a time, and in the complement this corresponds to deleting circuits from a complete $d$-clutter on vertex set $[n]$. For the case of $d=2$ we explicitly state the corollary.

\begin{cor} \label{cor:graph}
Suppose $G$ is a graph and let $e = ij$ be an edge in $G$.  Then $e = ij$ is an exposed edge if and only if $x_ix_j$ is a linear divisor for the ideal $I_{\overline G}$, where $\overline G$ is the complement of $G$.
\end{cor}

Combining this with the results of \cite{HerHibZhe} we obtain another proof of the result from \cite{CulGurSti} mentioned in the introduction.  Recall that by definition an edge erasure is the result of removing an edge $e$ that is \emph{properly} exposed.  For the case of the graphs this characterizes (complements of) chordal graphs that are \emph{connected}.

 \begin{cor}\cite[Theorem~8]{CulGurSti} \label{cor:erasures}
A graph $G$ can be obtained from a complete graph through a sequence of removing exposed edges if and only if $G$ is chordal.  In this case $G$ is connected if and only if each removal is an edge erasure.
\end{cor}

\begin{proof}
From Theorem \ref{thm:main} we have that removing an exposed edge $e = ij$ from a graph $H$ corresponds to adding the generator $x_ix_j$ that is a linear divisor in $I_{\overline H}$.  Hence performing a sequence of edge erasures on a complete graph to obtain a graph $G$ results in an ideal $I_{\overline G}$ that has linear quotients.  An arbitrary (monomial) ideal with linear quotients has a linear resolution, and hence in this case $G $ (the complement of ${\overline G}$) is chordal by Fr\"oberg's Theorem. On the other hand we know that if $G$ is chordal we have that $I_{\overline G}$ has a linear resolution. In \cite{HerHibZhe} is it shown that edge ideals with linear resolutions in fact have linear quotients.  Suppose $(g_1, g_2, \dots, g_k)$ is some linear quotient ordering for the ideal $I_{\overline G}$, so that for all $i = 2, \dots, g$ we have that $g_i$ is a linear divisor for the ideal $\langle g_1, \dots, g_{i-1}\rangle$.  From Theorem \ref{thm:main} this implies that the edge corresponding to $g_1$ is exposed in $K_n$, that $g_2$ is exposed in $K_n \backslash g_1$, etc.  Hence $G$ can be obtained by a sequence of removing exposed edges starting from a complete graph.

Next suppose that $G$ is obtained from the complete graph $K_n$ via a sequence of removing exposed edges. Let $K_n = H_1, H_2, \dots, H_k = G$ be the sequence of graphs obtained by removing these exposed edges, so that $e_i$ is an exposed edge in $H_i$. We claim that $G$ is connected if and only if each exposed edge was in fact \emph{properly} exposed.  For one direction, suppose that $G$ is disconnected.  Since we obtained $G$ from removing edges from $K_n$ then at some point in the process of removing exposed edges the sequence of graphs first became disconnected.  Suppose $H_{j+1}$ is the first graph is the sequence that is disconnected.  This implies that $e_j = v_1v_2$ is a bridge edge, and in particular not contained in any cycle in $H_j$. This implies that $e_j$ is not contained in any larger clique and hence is not properly exposed. For the other direction suppose one of the edges $e= vw$ in the deletion sequence was not properly exposed.  We claim that removing this edge results in a disconnected graph. If not, there must be another path in the graph that connects the vertices $v$ and $w$, say $v = v_1, v_2, \dots, v_k =w$.  Choose this path to be of minimum length.  If $k = 3$, the the set $\{v, v_2, w\}$ forms a clique, a contradiction to the assumption that $e$ was not properly exposed.  But if $k > 4$ then the vertices $v_1, v_2, \dots,  v_k$ forms a $k$-cycle, which must have a chord since the underlying graph is chordal.  This chord provides a shorter path from $v$ to $w$, a contradiction to the choice of $v_1, \dots, v_k$. We conclude that the graph in fact became disconnected by removing $e$. The result follows.
\end{proof}

\begin{figure}[h]
	\includegraphics[scale = 0.415]{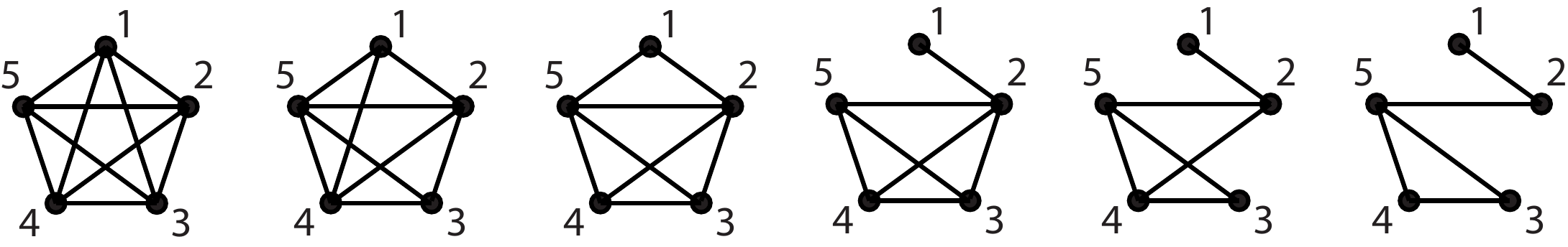}
	\caption{A sequence of erasures resulting in the graph $G$.  For example in the last step, the edge $24$ is contained in the maximal clique $245$, and the relevant colon ideal is given by $(I_4:x_2x_4) = \langle x_1, x_3 \rangle$.}
\end{figure}

We wish to generalize the connectivity condition for graphs to the context of $d$-clutters for $d > 2$.  For this we use the following algebraic characterization of connectivity.  Here $\pdim$ refers to the \emph{projective dimension} of the underlying module.

\begin{lemma}
A graph $G$ on vertex set $[n]$ is connected if and only if
\[\pdim(I_{\overline G}) < n-2.\]
\end{lemma}

\begin{proof}
We first observe that $I_{\overline{G}}$ is the Stanley-Reisner ideal of the simplicial complex $\Delta(G)$, where $\Delta(G)$ is the clique complex of $G$ (the simplicial complex whose faces are complete subgraphs of $G$).  We then employ Hochster's formula (see for instance \cite{MilStu}), which describes the Betti numbers of a Stanley-Reisner ideal in terms of the homology of induced complexes:
\[\beta_{i,j} = \sum \dim_{\mathbb K} \tilde H_{j-i-2} (\Delta(S); {\mathbb K}),\]
\noindent
where the sum is over all $j$-subsets $S \subset [n] = V(G)$, and $\Delta(S)$ denotes the clique complex of the graph induced on the vertex set $S$. 

Recall that the projective dimension of $I_{\overline{G}}$ is the largest $i$ such that $\beta_{i,j} \neq 0$ for some $j$. If $G$ is disconnected then we have $\tilde H_0 (\Delta(G),{\mathbb K}) \neq 0$, so that $\beta_{n-2 ,n} \neq 0$ and hence $\pdim(I_{\overline G}) \geq n-2$.  On the other hand if $\pdim(I_{\overline{G}}) \geq n-2$ then by Hochster's formula we must have $\beta_{n-2,n} \neq 0$ so that $\tilde H_0 (\Delta(G)) \neq 0$, which implies that $G$ is disconnected.
\end{proof}

For general $d$-clutters we have an an analogous statement.  We begin with a definition.

\begin{defn}
Suppose ${\mathcal C}$ is a $d$-clutter with complement circuit ideal $I_{\overline{\mathcal C}}$.  Suppose $e \in {\mathcal C}$ has the property that ${\bf x}_e$ is a linear divisor for $I_{\overline{\mathcal C}}$, and let $I_e = \langle I_{\overline{\mathcal C}} \cup \{{\bf x}_e\} \rangle$.  Define the \emph{Betti contribution} of ${\bf x}_e$ to be the set
\[\{i \in {\mathbb N}: \beta_{i} (I_{\overline{\mathcal C}}) \neq \beta_{i} (I_e)\}.\]
\noindent
From Equation \ref{eq:betti} we have that the Betti contribution has the form $\{0,1, \dots, k\}$ for some integer $k$.  We say that the Betti contribution is \emph{small} if $k < n-d$.
\end{defn}

 \begin{prop}\label{prop:small}
 	Suppose ${\mathcal C}$ is a $d$-clutter and $e \in {\mathcal C}$ is an exposed circuit.  Then $e$ is properly exposed if and only if the Betti contribution of ${\bf x}_e$ is small.
	 \end{prop}
 
 \begin{proof}
Let $I_{\overline{\mathcal C}}$ denote the circuit ideal of the complement of ${\mathcal C}$, and let $I_e$ denote the ideal obtained from adding the generator ${\bf x}_e$.  Suppose the colon ideal $(I_{\overline{\mathcal C}} : {\bf x}_e)$ is given by $\langle x_{i_1}, x_{i_2}, \dots, x_{i_\ell} \rangle$. As discussed in Section \ref{sec:definitions}, a minimal resolution of $I_e$ is obtained by taking the mapping cone of $f:{\mathcal K}_{\ell} \rightarrow {\mathcal F}$, where ${\mathcal K}_\ell$ is a Koszul resolution on $\ell$ generators and ${\mathcal F}$ is a minimal resolution of $I_{\overline{\mathcal C}}$. From Equation \ref{eq:betti} we see that the largest element in the Betti contribution of ${\bf x}_e$ is $\ell$. 

Suppose the edge $e$ is uniquely contained in the maximal clique $K$. From Theorem \ref{thm:main} we have that $x_{i_j} \in (I_{\overline{\mathcal C}} : {\bf x}_e)$ if only if the vertex $i_j$ satisfies $i_j \notin K$.  If $e$ is itself a maximal clique of the $d$-clutter ${\mathcal C}$ then we have $n-d$ vertices in the complement and hence by Equation \ref{eq:betti} we have that the Betti contribution of ${\bf x}_e$ has value $n-d$.  On the other hand if $|K| > d$ (so that $e$ is strictly contained in $K$) then we have at most $n-d-1$ vertices in the complement, in which case the Betti contribution is small.
\end{proof}
 

From this we get the desired analogue of Corollary \ref{cor:erasures} in the setting of $d$-clutters.  Once again recall that a circuit erasure is the removal of a circuit that is \emph{properly} exposed.

 \begin{prop}\label{prop:cluttererasure}
 	A $d$-clutter ${\mathcal C}$ can be obtained from a complete $d$ clutter $K_n^d$ through a sequence of circuit erasures if and only if $I_{\overline{\mathcal C}}$ has linear quotients and 
 	\[\pdim(I_{\overline {\mathcal C}}) < n-d.\]
 \end{prop}

If ${\mathcal C}$ is a $d$-clutter obtained by removing exposed circuits from $K_n^d$, we see from the proof of Proposition \ref{prop:small}  that $\beta_{n-d+1}$ counts the number of circuits that were not properly exposed in the removal process.   The other Betti numbers are similarly controlled by the cardinalities of cliques in the removal process, as the next result spells out (see also \cite[Corollary 8.2.2]{HerHib} for a similar observation in the purely algebraic setting).

\begin{cor}\label{cor:Betti}
Suppose ${\mathcal C}$ is a $d$-clutter on vertex set $[n]$ obtained from $K_n^d$ by removing a sequence of exposed circuits $(e_1, e_2, \dots, e_r)$.  Let $(K_n^d = {\mathcal C}_1, {\mathcal C}_2, \dots, {\mathcal C}_{k+1} = {\mathcal C})$ be the sequence of $d$-clutters obtained in this process, so that $e_j$ is exposed in ${\mathcal C}_j$. By definition each $e_j$ is contained in a unique maximal clique $K_j \subset {\mathcal C}_j$, let $k_j = n-|K_j|$. Then the Betti numbers of $I_{\overline {\mathcal C}}$ are given by 
\[\beta_i(I_{\overline {\mathcal C}}) = \sum_{j=1}^r {k_j \choose i}.\]
\end{cor}

\begin{proof}
From Theorem \ref{thm:main} we have that each time we add the generator ${\bf x}_{e_j}$ we glue on a Koszul resolution on $k_j$ generators to the desired minimal resolution.  The result then follows from Equation \ref{eq:betti}.
\end{proof}

\begin{rem}
The multiset $\{ k_j \}_{j=1}^n$ described in Corollary \ref{cor:Betti} is an invariant of the $d$-clutter ${\mathcal C}$, and in fact can be seen to coincide with the $h$-vector of a certain simplicial complex obtained from ${\mathcal C}$.  Namely, let ${\mathcal S}({\mathcal C})$ denote the simplicial complex whose facets are given by $\{F_1, F_2, \dots, F_k\}$, where $F_j = [n] - e_j$. From Proposition \ref{prop:lin} we see that ${\mathcal S}({\mathcal C})$ is a shellable simplicial complex, with a shelling order $F_1, F_2, \dots, F_k$ corresponding to the order of removing exposed circuits.  For each $j$ the restricted set ${\mathcal R}(F_j)$ has the property that $|{\mathcal R}(F_j)| = k_j$ (see the definition and discussion in the next section). Hence from \cite[Proposition 2.3]{Sta} we have that
\[h_i({\mathcal S}({\mathcal C})) = |\{j: k_j = i\}|.\]
\end{rem}

Our result is similar in spirit to a formula for the chromatic polynomial of a chordal graph obtained from its description as a sequence of simplicial vertices.  To recall this connection suppose $(v_1, v_2, \dots, v_n)$ is an ordering of the vertices of $G$ with the property that $N_i(v_i)$ is a complete graph, where $N_i(v_i)$ denotes the neighborhood of $v_i$ in the subgraph of $G$ induced by vertex set $\{v_i, v_{i+1}, \dots, v_n\}$.  For each $i$ let $d_i$ denote the number $|N_i(v_i)| - 1$. Then one can show (see for instance \cite[Remark 2.5]{Agn}) that the chromatic polynomial of $G$ is given by
\[\chi_t(G) = \prod_{i=1}^n (t-d_i).\]
\noindent
We do not know if there are other combinatorial interpretations of the $k_j$.

We note that there is a geometric interpretation of the Betti numbers of (certain) ideals with linear quotients in terms of the face numbers of certain polyhedral complexes supporting a cellular resolution.  We refer to \cite{DocMoh} for details but note that in our running example (Example \ref{ex:chordal} from above) a minimal resolution of $I_{\overline G}$ 
\[0 \leftarrow I_{\overline{G}} \leftarrow R^5 \leftarrow R^6 \leftarrow R^2 \leftarrow 0.\]

\noindent
is supported on the polyhedral complex depicted below.   The complex has five vertices, six edges, and two 2-cells.  Also note that $\pdim(I_{\overline G}) = 2 < 5-2.$

\begin{figure}[h]
\includegraphics[scale = 1.05]{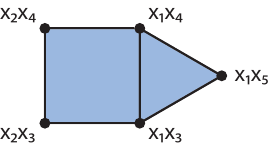}
\caption{A cellular complex supporting a minimal resolution of $I_{\overline G}$, where $G$ is the graph from Figure \ref{fig:example}.}
\end{figure}

We discuss one more higher-dimensional example to illustrate our constructions.

 \begin{example}Let $K^3_5$ denote the complete 3-clutter on 5 vertices.  We will remove circuits $K^3_5$ in the following order.  Here we suppress set brackets, so that $125 = \{1,2,5\}$.
 
 \begin{enumerate}
 	\item
 	Remove the circuit $125$, uniquely contained in the clique $12345$ and giving $k_1 = 0$.
 	\item
 	Remove $135$, uniquely contained in the clique $1345$ so that $k_2 = 1$ .
 	\item
 	Remove $145$, itself a clique and giving $k_3 = 2$.
	
	\end{enumerate}
 \noindent
 At this point we have a 3-clutter ${\mathcal C}$ that is geometrically a bipyramid over a triangle.  The corresponding complement circuit ideal is
 \[I_3 = \langle x_1x_2x_5, x_1x_3x_5, x_1x_4x_5 \rangle,\]
 which indeed has a linear resolution (see below).   From Corollary \ref{cor:Betti} we can compute Betti numbers
 \begin{equation}
 \begin{split}
 \beta_0 = {0 \choose 0} + {1 \choose 0} + {2 \choose 0} = 3 \\
 \beta_1 = {0 \choose 1} + {1 \choose 1} + {2 \choose 1} = 3 \\
 \beta_2 = {0 \choose 2} + {1 \choose 2} + {2 \choose 2} = 1.
\end{split}
\end{equation}
 
  \begin{figure}[h]
 	\includegraphics[scale = .35]{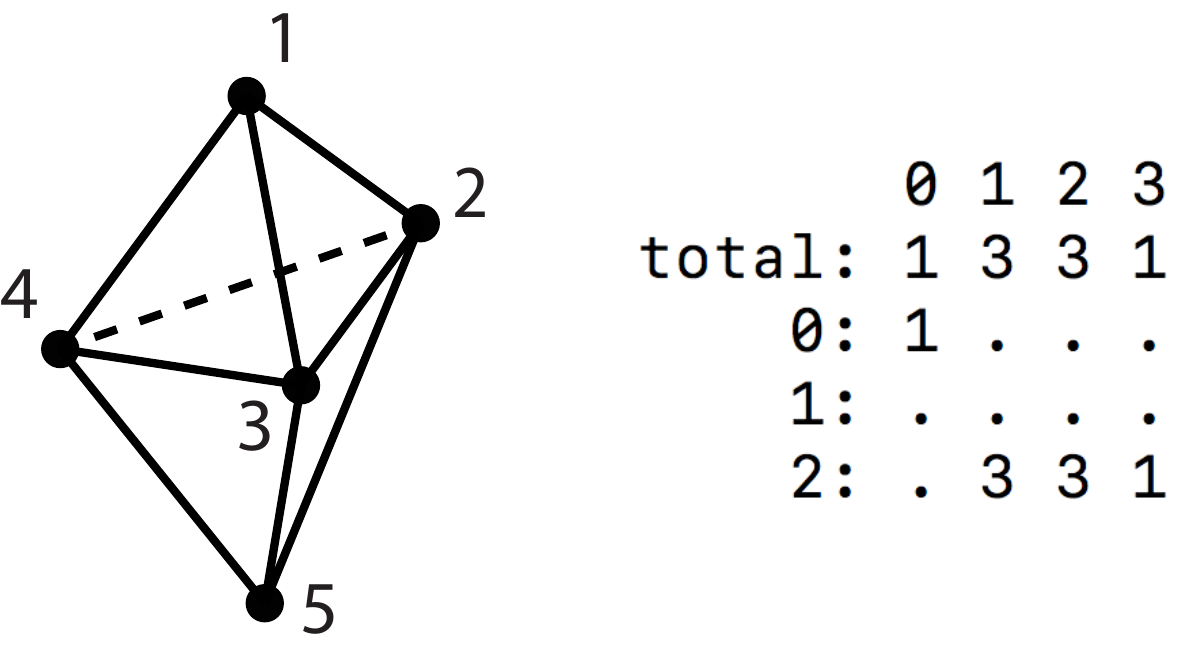}
 	\caption{A chordal 3-clutter ${\mathcal C}$, with the Betti table of $R/I_{\overline {\mathcal C}}$.}
 \end{figure}
 \noindent
 Note that $234$ is (properly) contained in cliques $1234$ \emph{and} $2345$.  Indeed if we remove $234$ we get
 \[I_4 =  \langle x_1x_2x_5, x_1x_3x_5, x_1x_4x_5 , x_2x_3x_4 \rangle,\]
 which has a nonlinear resolution.

 \begin{figure}[h]
 \includegraphics[scale = .39]{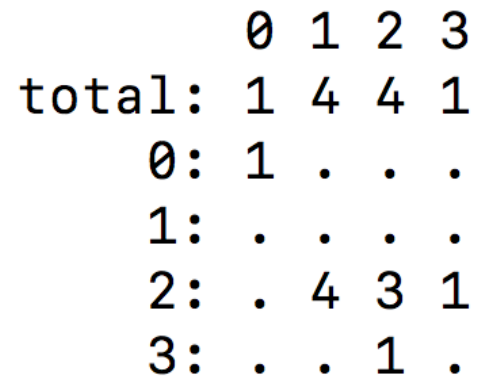}
 \caption{Betti table for $R/I_4$.}
 \end{figure}

Also note that in Step (3) we removed a circuit that was exposed but not properly exposed.  This is reflected by the fact that the corresponding ideal has projective dimension 2.  If in Step (3) we instead remove $123$ (which is uniquely contained in the clique $1234$) we obtain the `connected' 3-clutter depicted below. This clutter has the property that if we include the complete 1-skeleton the resulting simplicial complex has vanishing first homology.
 
   \begin{figure}[h] \label{fig:connclutter}
 	\includegraphics[scale = .35]{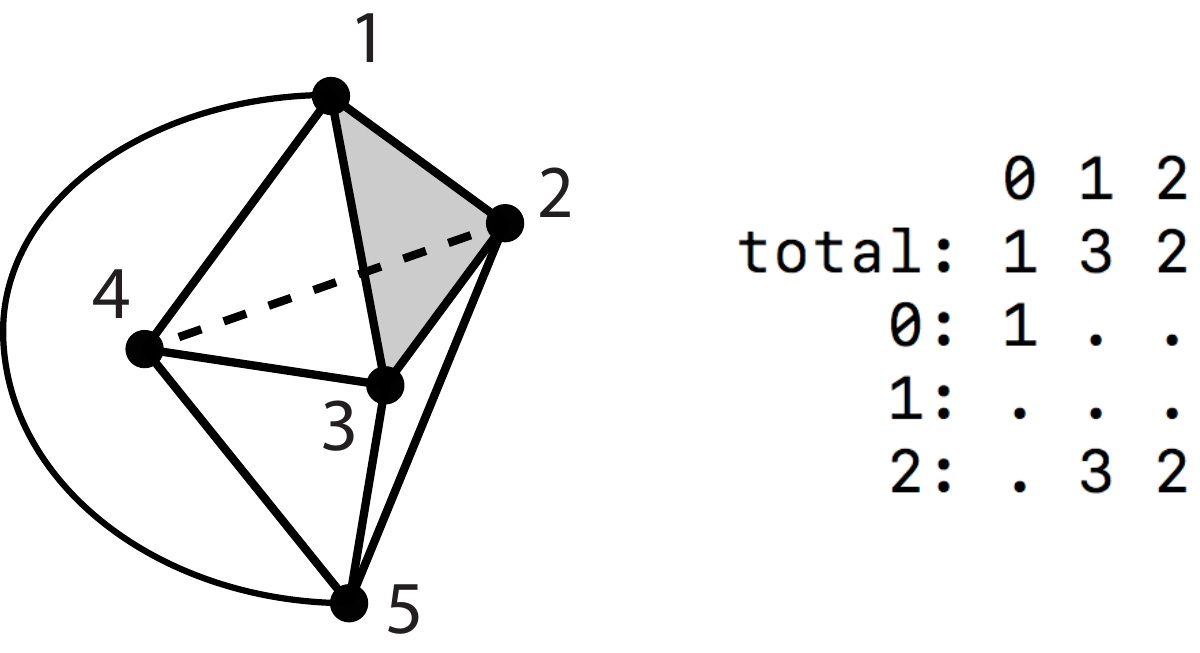}
 	\caption{A `connected' chordal 3-clutter ${\mathcal D}$, with the Betti table for $R/{\overline {\mathcal D}}$. The clutter consists of all $3$-subsets of $[5]$ except $123$, $125$, and $135$.}
	\end{figure}

  \end{example}
  
\begin{rem}
 The concept of an exposed circuits is reminiscent of certain constructions from the study of \emph{simple homotopy theory} (see for example \cite{Coh}).  Here if $\Delta$ is a simplicial complex, a face $\tau \in \Delta$ is called a \emph{free face} if it is contained in a unique facet $\sigma$.   The removal of $\tau$ along with all simplices $\gamma$ such that $\tau \subset \gamma \subset \sigma$ is called an \emph{elementary collapse}.  Since elementary collapses preserve (simple) homotopy type, in particular any such complex obtained this way from a simplex will be contractible.

If ${\mathcal C}$ is a $d$-clutter on $[n]$ we define the \defi{$d$-clique clutter} $\Delta_d({\mathcal C})$ to be the simplicial complex on the same vertex set $[n]$ with 

\begin{itemize}
\item a complete $(d-2)$-skeleton,
\item $(d-1)$-dimensional faces corresponding to the circuits of ${\mathcal C}$,
\item for $k \geq d$, all $k$-faces $\sigma$ such that all $d$-subsets of $\sigma$ are circuits in ${\mathcal C}$.  
\end{itemize}

For example if $d=2$ this is the usual clique complex of a graph. One can see that a properly exposed circuit $e$ in a $d$-clutter ${\mathcal C}$ corresponds to a free face (of dimension $d-1$) in the simplicial complex $\Delta_d({\mathcal C})$. Hence if ${\mathcal C}$ is a clutter obtained by performing a sequence of circuit erasures starting with $K_n^d$ we see that $\Delta_d({\mathcal C})$ is contractible.  In the graph case ($d=2$) this is equivalent to the fact that removing an exposed edge disconnects the graph if and only if it is not properly exposed, as in the proof of \ref{cor:erasures} (recall that the clique complex of a chordal graph is contractible if and only if it is connected).  We note that a further connection between chordality and simple homotopy theory (in the context of \emph{$d$-collapsibility}) has been recently explored by Bigdeli and Faridi \cite{BigFar}.  The authors of \cite{BigFar} use the language of Stanley-Reisner theory to generalize results from \cite{BPZ} to include the case of square-free monomial ideals that are not necessarily generated in a fixed degree.
\end{rem}

\section{Applications: Simon's conjecture and higher chordality} \label{sec:applications}

We next discuss other applications and corollaries of our results.   We also relate our study to other constructions of chordal complexes from the literature.

\subsection{Extendably shellable complexes} \label{sec:extend}
From Theorem \ref{thm:main} and Proposition \ref{prop:lin} we see that the process of removing exposed circuits from a $d$-clutter on vertex set $[n]$ is closely related to \emph{shellings} of (pure) simplicial complexes of dimension $n-d-1$.  We now discuss how our results from above can be applied in this context.   We begin with a definition.

\begin{defn}
A shellable complex $\Delta$ is said to be \defi{extendably shellable} if any shelling of a subcomplex of $\Delta$ can be extended to a shelling of $\Delta$.  
\end{defn}

Here a subcomplex of $\Delta$ is a simplicial complex $\Gamma$ on the same vertex set as $\Delta$, whose set of facets consists of a subset of the facets of $\Delta$. Ziegler \cite{Zie} has shown that there exist simple and simplicial polytopes whose boundary complexes are not extendably shellable. Simon \cite{Sim} has conjectured that every $k$-skeleton of a simplex is extendably shellable.   Our interpretation of results of \cite{CulGurSti} provides a proof of the conjecture in some special cases, and also leads to generalizations.  

Recall that if $\Delta$ is a pure $k$-dimensional simplicial complex $\Delta$ with shelling order of its facets $(F_1, F_2, \dots, F_{f_k})$, the \defi{restricted set} of the facet $F_i$, denoted ${\mathcal R}(F_i)$, is the unique minimal element of $\Delta_i \backslash \Delta_{i-1}$, where $\Delta_i$ is the subcomplex of $\Delta$ generated by $F_1, \dots, F_i$.  In other words we have 
\[{\mathcal R}(F_i) = \{\ell \in F_i: F_i \backslash \{\ell\} \in \Delta_{i-1}\}.\]

If $\Delta$ is a shellable complex with shelling order $F_1, F_2, \dots, F_r$ then from Proposition \ref{prop:lin} we have that the ideal $I_{\Delta^\vee} = \langle u_1, u_2, \dots, u_r \rangle$ has linear quotients with the given ordering of the generators, where $u_i = {\bf x}_{\overline{F_i}}$.  We have from \cite[Lemma 8.2.3]{HerHib} that the ideal $(u_1, \dots,u_{i-1}:u_i)$ is generated by the monomials $u_j/\gcd(u_j,u_i), j = 1, \dots, i-1$, and furthermore that $u_j/\gcd(u_j,u_i) = x_\ell$ for some $\ell \in [n]$. It follows that for all $i \in [n]$, the set ${\mathcal R}(F_i)$ is given by the (index set of the) variables generating the colon ideal $(I_{i-1}:u_i)$, where $I_{i-1} = \langle u_1, \dots,u_{i-1} \rangle$.

\begin{lemma}\label{lem:shelling}
A  sequence $e_1, e_2, \dots, e_k$ of removing exposed circuits from the complete $d$-clutter $K_n^d$ corresponds to a shelling sequence $F_1, F_2, \dots, F_k$ of the $(n-d-1)$-dimensional complex on vertex set $[n]$ whose facets are given $F_i = [n] \backslash e_i$.  A circuit $e_i$ is \emph{properly} exposed if and only if the restricted set of $F_i$ consists of less than $n-d$ elements.
\end{lemma}

\begin{proof}
Theorem \ref{thm:main} implies that each $e_i$ is an exposed circuit in $K_n^d - \{e_1, e_2, \dots, e_{i-1}\}$ if and only if ${\bf x}_{\bf e_i}$ is a linear divisor for the ideal $I_{i-1} = \langle {\bf x}_{e_1}, {\bf x}_{e_2}, \dots, {\bf x}_{e_{i-1}} \rangle$. Proposition \ref{prop:lin} then implies that $e_1, e_2, \dots, e_k$ is a sequence of exposed circuits if and only if $F_1, F_2, \dots, F_k$ is a shelling order for the simplicial complex it defines, where $F_i = \overline e_i = [n] \backslash e_i$.  As we have seen, the restricted set of $F_i$ corresponds to variables generating the colon ideal $(I_{i-1}: {\bf x}_{e_i})$, which by Theorem \ref{thm:main} is given by $[n] \backslash K$, where $K$ is the unique maximal clique containing the edge $e_i$.  Hence $|K| > d$ if and only if the restricted set of $F_i$ consists of less than $n-d$ elements.
\end{proof}

\begin{cor}\label{cor:ext}\cite{BPZDec}
The $(n-3)$-skeleton of a simplex on vertex set $[n]$ is extendably shellable.
\end{cor}

\begin{proof}
Let $\Delta_{n^-1}^{(n-3)}$ denote the $(n-3)$-skeleton of the simplex on vertex set $[n]$. Suppose $H$ is a shellable proper subcomplex of $\Delta_{n-1}^{(n-3)}$, with shelling order $F_1, F_2, \dots, F_h$.  Note that each $F_i$ is a subset of $[n]$ of size $n-2$. Lemma \ref{lem:shelling} implies that the graph on vertex set $[n]$ with edges $e_i = [n] \backslash F_i$ is obtained from the complete graph $K_n$ by removing exposed edges; hence it is chordal.   In \cite[Proposition 10]{CulGurSti} it is shown that if $G$ is any chordal graph then $G$ contains an exposed edge $e$ (in fact if $G$ is not a forest then this edge can be taken to be \emph{properly} exposed).  Hence we can extend the shelling sequence with the facet $F_{h+1} = [n] \backslash e$. Induction on ${n \choose 2} - h$ implies that any shelling of a subcomplex of $\Delta_{n-1}^{(n-3)}$ can be extended to an entire shelling.
\end{proof}

Note that the $(n-2)$-skeleton of the simplex on vertex set $[n]$ is the boundary of a simplex, which is clearly extendably shellable (in fact any sequence of facets constitutes a shelling). The $(n-1)$-skeleton is the simplex itself which consists of a single facet.  Hence we have the following corollary.

\begin{cor}\label{cor:shell}
For all $k \geq n-3$, the $k$-skeleton of a simplex on $[n]$ is extendably shellable.
\end{cor}

This result was also obtained in \cite{BPZDec} using methods related to another notion of chordal clutters (see the next section).  A more careful analysis of the results from \cite{CulGurSti} leads to another class of $(n-3)$-dimensional simplicial complexes that are extendably shellable.

\begin{prop}
Suppose $\Delta$ is an $(n-3)$-dimensional shellable simplicial complex on vertex set $[n]$. If $\Delta$ is contractible and has ${n \choose 2} - n + 1$ facets then it is extendably shellable.
\end{prop}
\begin{proof}
It is known (see for example \cite[Theorem 12.3]{ Koz}) that a shellable complex $\Delta$ is contractible if and only if any shelling of $\Delta$ there are no restricted sets of size $n-2$ (in fact $\Delta$ is either contractible or has the homotopy type of $\ell$ spheres of dimension $n-3$, where $\ell$ is the number of restricted sets of size $n-2$).  By Lemma \ref{lem:shelling} a sequence $e_1, e_2, \dots, e_k$ of removing properly exposed edges from $K_n$ corresponds to a shelling sequence $F_1, F_2, \dots, F_k$ of an $(n-3)$-dimensional complex, where at each step the restricted set consists of less than $n-2$ elements.  
 
Hence a contractible shellable complex with ${n \choose 2} - n +1$ facets corresponds to a connected graph with $n-1$ edges--in other words a tree $T$ on vertex set $[n]$.  Any shelling of a subcomplex $\Gamma$ of $\Delta$ corresponds to a connected chordal graph $G$ containing the tree $T$.  From \cite[Corollary 15]{CulGurSti} we see that if $H$ is an edge-weighted connected chordal graph then a minimal spanning tree of $H$ can be obtained by a sequence of edge erasures (removing \emph{properly} exposed edges).  Hence if we assign weighs to the edges of our graph $G$ so that $T$ is the only minimal spanning tree of $G$, we can obtain $T$ from $G$ via such a sequence. This implies that we can extend the shelling of $\Gamma$ to a shelling of $\Delta$, as desired.
\end{proof}

\begin{rem}
The method of finding a minimal spanning tree from \cite{CulGurSti} (and discussed above) is a variation of Kruskal's algorithm \cite{Kru}.  We have used this to show that any spanning tree of a connected chordal graph $G$ can be obtained via a sequence of deleting properly exposed edges.  One wonders if similar arguments can be employed to show that \emph{any} chordal subgraph $H$ of $G$ can be obtained from $G$ via sequence of removing exposed edges.  If true this would imply that \emph{any} $(n-3)$-dimensional shellable complex on vertex set $[n]$ is extendably shellable \footnote{This fact has since been established by the author and others, see \cite{CDGS}.}.
\end{rem}

Finally we end this section with a reformulation of Simon's conjecture in terms of exposed circuits.

  \begin{conj} [reformulation of Simon's conjecture]
  Suppose ${\mathcal C}$ is a $d$-clutter obtained from the complete $d$-clutter $K^d_n$ by a sequence of removing exposed circuits.  Then ${\mathcal C}$ contains an exposed circuit.
  \end{conj}

\subsection{Chordal complexes in higher dimensions}\label{sec:higher}

As we have seen the process of removing exposed edges from a complete $d$-clutter gives rise to a circuit ideal $I_{\overline{\mathcal C}}$ that has linear quotients.  For the case $d=2$ this in fact characterizes chordal graphs.  Hence our constructions give rise to a natural candidate for what might be considered a `chordal $d$-clutter' (or more precisely the complement of one). In recent years several authors have studied various higher-dimensional generalizations of chordal graphs, many inspired by Fr\"oberg's Theorem in an attempt to give a combinatorial characterization of squarefree monomial ideals having a $d$-linear resolution over any field.  By Hochster's formula this is equivalent to restricting the topology of induced subcomplexes, although one hopes for a more global description.  For the reader's convenience we briefly review some of these approaches below.

In one attempt to define a chordal complex, the notion of a `chordless cycle' is generalized to higher-dimensional settings.  This is the approach taken by Connon and Faridi in \cite{ConFar}, in which they give a combinatorial description of a $(d-1)$-dimensional cycle as a $d$-clutter ${\mathcal C}$ that is strongly connected and such that the `degree' of each ridge is even.  The authors introduce notions of `chordless' cycles and for instance show that if ${\mathcal C}$ is a $d$-clutter such that $I_{\overline{\mathcal C}}$ admits a linear resolution over any field, then ${\mathcal C}$ is `orientably-cycle-complete'.  As a partial converse, they show that the clutter ideal of the complement of a \emph{$d$-tree} (a $d$-clutter with no cycles) has a linear resolution over any field of characteristic 2.  In \cite{AdiNevSam} a more homological approach is taken to study notions of higher chordality.

In other attempts to generalize chordal graphs the notion of a `simplicial vertex' is taken as the starting point.  This is the approach taken by Emtander in \cite{Emt} where the notion of a vertex $v \in {\mathcal C}$ with a \emph{complete-neighborhood} is introduced.  A $d$-uniform clutter is then `chordal' in this context if every induced subclutter admits a vertex with a complete neighborhood (or else has no circuits).  Woodroofe \cite{Woo} takes a related but independent approach, defining a notion of a \emph{simplicial vertex} $v \in {\mathcal C}$ that borrows from the circuit characterization of matroids. A (not necessarily uniform) clutter is then said to be `chordal' in this context if every \emph{minor} of ${\mathcal C}$ admits a simplical vertex. 

In yet another direction Bigdeli, Yazdan Pour, and Zaare-Nahandi \cite{BPZ} employ the notion of a simplicial \emph{ridge} to study chordality.  Recall that a \defi{ridge} in a $d$-clutter ${\mathcal C}$ is a set $R$ of vertices of size $d-1$ such that $R \subset e$ for some circuit $e \in {\mathcal C}$ (note that in the setting of graphs a vertex is also a ridge).  From \cite{BPZ} a ridge is said to be \defi{simplicial} if the induced subclutter on $R \cup \{v \in [n]: R \cup \{v\} \in {\mathcal C} \}$ is the complete $d$-uniform clutter.  A clutter ${\mathcal C}$ is then `chordal' in this context if there exists a sequence of ridges $R_1, R_2, \dots, R_k$ of ${\mathcal C}$ such that $R_i$ is simplicial in the clutter ${\mathcal C} - (R_1 \cup \cdots \cup R_{i-1})$, and ${\mathcal C} - (R_1 \cup \cdots \cup R_k) = \emptyset$.  To distinguish this notion from others we will call such a clutter \defi{BYZ-chordal} in what follows.  In \cite{BHPZ} it is shown that if ${\mathcal C}$ is BYZ-chordal then the ideal $I_{\overline{\mathcal C}}$ has a linear resolution over every field. In fact the class of BYZ-chordal $d$-clutters includes the classes of chordal clutters described above.   There do, however, exist monomial squarefree ideals with a linear resolution over every field that do not arise as complements of BYZ-chordal clutters (for example the clutter ideal coming from a certain triangulation of a dunce hat).  There also exist ridge chordal $d$-clutters ${\mathcal C}$ with the property that $I_{\overline{\mathcal C}}$ does not have linear quotients.  As far as we know the following question is still open \footnote{Since posting this paper, a counterexample to this statement has been found by Benedetti and Bolognini, see \cite{BenBol}.}.

\begin{conj} \label{conj:ridge}
If ${\mathcal C}$ is a $d$-clutter with the property that $I_{\overline{\mathcal C}}$ has linear quotients, then ${\mathcal C}$ is BYZ-chordal.
\end{conj}

\begin{rem}
Our constructions are also related to the notion of BYZ-chordality from \cite{BPZ} as follows.  If ${\mathcal C}$ is a $d$-clutter on vertex set $[n]$, let ${\mathcal C}^{d+1}$ denote the $(d+1)$-clutter on the same vertex set, with circuits given by all cliques in ${\mathcal C}$ of size $d+1$.   For example if $d=2$ (so that ${\mathcal C}$ is a graph) then ${\mathcal C}^{d+1}$ consists of all triangles in the underlying graph.
One can then show that $e$ is an exposed cricuit of ${\mathcal C}$ if and only if $e$ is a simplicial ridge of ${\mathcal C}^{d+1}$.  We thank Mina Bigdeli for pointing this out to us \cite{BigPer}. This observation also provides an alternative proof of a result in \cite{NikZaa}, where it is shown that for a graph $G$ a sequence of edges $e_1, \dots, e_t$ is a simplicial sequence of ridges in $G^3$ if and only if the ideal $I_{\overline G}$ has linear quotients.  

Conversely we note that if ${\mathcal C}$ is a $d$-clutter with an exposed edge $e$ it is \emph{not} necessarily the case that we can find an element $e$ such that $e \backslash \{v\}$ is a simplicial ridge of ${\mathcal C}$.  
\end{rem}

\section{Final Remarks}

As we mentioned above, the study of edge erasures in chordal graphs developed in \cite{CulGurSti} was originally motivated by questions involving clustering algorithms and in particular generalizations of single-linkage clustering.  Finding a minimal spanning tree of a weighted complete graph (equivalently a finite metric space) provides the basis for single-linkage clustering. The hope was that minimal chordal graphs may serve a similar role for more general clustering algorithms that allow \emph{overlaps}.   It is not clear if the chordal $d$-clutters discussed here might have any relevance to these constructions. For this one might want a generalization of a metric space where $d$-tuples of points are assigned a `distance'.

In the process of generalizing Kruskal's algorithm for finding minimal spanning trees, the authors of \cite{CulGurSti} use the following property of properly exposed edges in a chordal graph.

 \begin{thm}\cite[Theorem~12]{CulGurSti} \label{thm:cycles}
 Suppose $G$ is a chordal graph, and let $\partial G$ denote the edge-induced subgraph of $G$ determined by the properly exposed edges of $G$.  Then every connected component of $\partial G$ is $2$-edge connected.
  \end{thm}
 
We do not know if something similar holds in the context of higher-dimensional $d$-clutters.  Is it the case that properly exposed circuits in a chordal $d$-clutter are also contained in some version of higher-dimensional cycles? One also wonders if there is an interpretation of Theorem \ref{thm:cycles} in the context of commutative algebra.

Another natural question to ask is if the generalization of Kruskal's algorithm that relies on Theorem \ref{thm:cycles} can be generalized to the context of higher dimensional complexes or more general matroids.  In particular given a circuit-weighted $d$-clutter ${\mathcal C}$ can one find a minimal `spanning tree' by removing properly exposed faces?  Note that the connectivity of such a spanning tree is not simply vanishing of the top $(d-1)$-homology of the simlicial complex $\Delta$ defined by the circuits (thought of as facets of $\Delta$), but also vanishing of the $(d-2)$ homology of the complex obtained by also including its complete $(d-2)$-skeleton (see Figure \ref{fig:connclutter}).

Finally, one wonders what role of edge/circuit weighted graphs and $d$-clutters might play in combinatorial commutative algebra. For instance if we assign weights to all quadratic squarefree monomials $x_ix_j$ in the polynomial ring $R = {\mathbb K}[x_1, \dots, x_n]$ our discussion above implies that among all  `minimal' quadratic monomial ideals $I$ with ${n \choose 2} - n + 1$ generators satisfying $\pdim I < n-2$, we can find such an ideal with the property that $I$ has linear quotients.  Here the weight of a monomial ideal is the sum of the weights of its generators (in its minimal set of generators).

\end{document}